\documentclass[12pt,reqno]{amsart}
\usepackage{amsmath}
\usepackage{amssymb}
\usepackage{algorithm}
\usepackage{algorithmic}

\def\B'c{{\mathcal{B'}}}
\def\U'c{{\mathcal{U'}}}

\def\opn#1#2{\def#1{\operatorname{#2}}}
\opn\chara{char}
\opn\length{\ell}
\opn\projdim{proj\,dim}
\opn\injdim{inj\,dim}
\opn\ini{in}
\opn\rank{rank}
\opn\depth{depth}
\opn\height{ht}
\opn\embdim{emb\,dim}
\opn\codim{codim}

\opn\Tr{Tr}
\opn\bigrank{big\,rank}
\opn\superheight{superheight}\opn\lcm{lcm}
\opn\trdeg{tr\,deg}\opn\reg{reg}
\opn\lreg{lreg}
\opn\set{set}
\opn\supp{Supp}
\opn\shad{Shad}
\opn\del{del}
\opn\div{div}
\opn\Div{Div}
\opn\cl{cl}
\opn\Cl{Cl}
\opn\Spec{Spec}
\opn\Supp{Supp}
\opn\supp{supp}
\opn\Sing{Sing}
\opn\Ass{Ass}
\opn\Ann{Ann}
\opn\Rad{Rad}
\opn\Soc{Soc}
\opn\Ker{Ker}
\opn\Coker{Coker}
\opn\Im{Im}
\opn\Hom{Hom}
\opn\Tor{Tor}
\opn\Ext{Ext}
\opn\End{End}
\opn\Aut{Aut}
\opn\id{id}

\opn\nat{nat}
\opn\GL{GL}
\opn\SL{SL}
\opn\mod{mod}
\opn\ord{ord}
\opn\aff{aff}
\opn\con{conv}
\opn\relint{relint}
\opn\st{st}
\opn\lk{lk}
\opn\cn{cn}
\opn\core{core}
\opn\vol{vol}
\opn\gr{gr}

\def\pot#1#2{#1[\kern-0.28ex[#2]\kern-0.28ex]}

\opn\dirlim{\underrightarrow{\lim}}
\opn\invlim{\underleftarrow{\lim}}

\def\pnt{{\raise0.5mm\hbox{\large\bf.}}}

\def\Implies{\ifmmode\Longrightarrow \else
     \unskip${}\Longrightarrow{}$\ignorespaces\fi}
\def\implies{\ifmmode\Rightarrow \else
     \unskip${}\Rightarrow{}$\ignorespaces\fi}
\def\iff{\ifmmode\Longleftrightarrow \else
     \unskip${}\Longleftrightarrow{}$\ignorespaces\fi}

\let\:=\colon
\newtheorem{Theorem}{Theorem}[section]
\newtheorem{Lemma}[Theorem]{Lemma}
\newtheorem{Corollary}[Theorem]{Corollary}
\newtheorem{Proposition}[Theorem]{Proposition}
\newtheorem{Remark}[Theorem]{Remark}

\let\epsilon=\varepsilon
\let\phi=\varphi
\let\kappa=\varkappa
\begin{document}
\title{On the Hilbert series of vertex cover algebras of unmixed
bipartite graphs}
\author{Cristian Ion}
\address{Faculty of Mathematics and Computer Science, Ovidius University,
Bd.\ Mamaia 124, 900527 Constanta, Romania,}
\email{cristian.adrian.ion@gmail.com}
\maketitle

\begin{abstract}
We compute the reduced Gr\"{o}bner basis of the toric ideal with respect 
to a suitable monomial order and we study the Hilbert series of the vertex
cover algebra $A(G)$, where $G$ is an unmixed bipartite graph
without isolated vertices. \\


\end{abstract}

\section*{Introduction}
Let $G=(V,E)$ be a simple (i.e., finite, undirected, loopless and without
multiple edges) graph with the vertex set $V=[n]$ and the edge set $E=E(G)$. 
For $k\in \mathbb{N}$, a $k$-vertex cover of $G$ is a vector 
$c=(c_{1},c_{2},...,c_{n})\in \mathbb{N}^{n}$ such that $c_{i}+c_{j}\geq k$ 
for every edge $\{i,j\}$ of $G$. 

Let $R=K[x_{1},x_{2},...,x_{n}]$ be the polynomial ring over a field $K$. The 
vertex cover algebra $A(G)$ is defined as the subalgebra of the one variable 
polynomial ring $R[t]$ generated by all monomials $x_{1}^{c_{1}}x_{2}^{c_{2}}...x_{n}^{c_{n}}t^k$, 
where $c=(c_{1},c_{2},...,c_{n})$ is a $k$-vertex cover of $G$. This algebra was 
introduced and studied in \cite{HerHib3}. Let $\mathfrak{m}$ be the maximal graded 
ideal of $R$. $\bar{A}(G)=A(G)/\mathfrak{m}A(G)$ is called the basic cover algebra 
and it was studied in \cite{HerHib2}, \cite{BenCon} and \cite{BerMic}. In \cite{Ion}, the Hilbert 
series of $A(G)$, $H_{A(G)}(z)$, for a Cohen-Macaulay bipartite graph $G$ is studied 
and several consequences are derived.

Our main aim in this paper is to extend the study of $H_{A(G)}(z)$ for unmixed 
bipartite graphs. It will turn out that many of the results concerning the 
Cohen-Macaulay case extend naturally to the larger class of unmixed bipartite graphs. 
The first step in getting the formula for $H_{A(G)}(z)$ is to compute the toric ideal 
of $A(G)$. This is done in Section 2. 

In the last section we state the main theorem which relates $H_{A(G)}(z)$ to the 
Hilbert series of the basic cover algebras $\bar{A}(G_{F})$, for all 
$F\subset \lbrack n]$. Based on this formula we derive sharp bounds for the 
multiplicity of $A(G)$.

\section{The lattice associated to an unmixed bipartite graph}
Let $G$ be an unmixed bipartite graph without isolated vertices. By \cite[Theorem 1.1]{Vil}  
we may assume that $G$ admits a bipartition of its vertices $V_{n}=W\cup W^{\prime }$, where $W=\{x_{1},...,x_{n}\}$ and $W^{\prime }=\{y_{1},...,y_{n}\}$, $n\geq 1$, such that:

\begin{itemize}
\item[(a)] $\{x_{i},y_{i}\}\in E(G)$, for all $i\in \lbrack n]$;

\item[(b)] if $\{x_{i},y_{j}\}\in E(G)$ and $\{x_{j},y_{k}\}\in E(G_{n})$,
then $\{x_{i},y_{k}\}\in E(G_{n})$, for all distinct $i,j,k\in \lbrack n]$.
\end{itemize}

Throughout this paper, whenever we refer to an unmixed bipartite graph we
assume it is given with its above bipartition.

For $\emptyset \neq F\subset \lbrack n]$ we denote by $G_{F}$ the subgraph of $G$
induced by the subset $V_{F}=\{x_{i}|i\in F\}\cup \{y_{i}|i\in F\}$.

\begin{Remark} \emph{$G_{F}$ also satisfies $(a)$ and $(b)$, hence $G_{F}$ is an unmixed bipartite graph 
on $V_{F}$ and each minimal vertex cover of $G_{F}$ has the cardinality equal to 
$\left\vert F\right\vert $.}
\end{Remark}

Let $K_{\{i,j\}}$, $1\leq i<j\leq n$, be the complete bipartite graph on 
$\{x_{i},x_{j}\}\cup \{y_{i},y_{j}\}$.

\begin{Lemma}
\label{vertcov} Let $G$ be an unmixed bipartite graph on $V_{n}=W\cup W^{\prime }$,  
$n\geq 2$. Suppose that $G$ has an induced 
subgraph $K_{\{i,j\}}$ with $1\leq i<j\leq n$. Let $H$ be the subgraph of $G$ 
induced by the subset $V_{n}\backslash \{x_{j},y_{j}\}$. Then there exists a
one-to-one correspondence between the sets $\mathcal{M}(G)$, respectively  
$\mathcal{M}(H)$, of minimal vertex covers of $G$, respectively $H $. More precisely, 
for all subsets $C\subset V_{n}\backslash \{x_{j},y_{j}\}$ we have:

\begin{itemize}
\item[(i)] if $x_{i}\in C$, then ${C\in }$ $\mathcal{M}(H)\Leftrightarrow
C\cup \{x_{j}\}\in \mathcal{M}(G)$;

\item[(ii)] if $x_{i}\not\in C$, then ${C\in }$ 
$\mathcal{M}(H)\Leftrightarrow C\cup \{y_{j}\}\in \mathcal{M}(G)$.
\end{itemize}
\end{Lemma}

\begin{proof}
Let $C\in \mathcal{M}(H)$. If $x_{i}\in C$, put $B=C\cup \{x_{j}\}$. We
show that $B\in \mathcal{M}(G)$. $B\cap \{x_{k},y_{l}\}\neq \emptyset $%
, for all $\{x_{k},y_{l}\}\in E(H)$ and $B\cap \{x_{j},y_{l}\}\neq \emptyset 
$, for all $\{x_{j},y_{l}\}\in E(G)$. Let $\{x_{k},y_{j}\}\in E(G)$ with $%
k\notin \{i,j\}$. Since $\{x_{k},y_{j}\}\in E(G)$ and $\{x_{j},y_{i}\}\in
E(G)$, it follows, by $(b)$, that $\{x_{k},y_{i}\}\in E(G)$. Hence $%
\{x_{k},y_{i}\}\in E(H)$ and $C\cap \{x_{k},y_{i}\}\neq \emptyset $. $C\in 
\mathcal{M}(H)$ implies that $\left\vert C\cap \{x_{i},y_{i}\}\right\vert
=1 $ and, since $x_{i}\in C$, we have that $y_{i}\notin C$. On the other
hand, $C\cap \{x_{k},y_{i}\}\neq \emptyset $, hence $x_{k}\in C$. Thus $%
B\cap \{x_{k},y_{i}\}\neq \emptyset $ and $B\in \mathcal{M}(G)$. If $%
y_{i}\in C$, put $B=C\cup \{y_{j}\}$. Similarly, it can be proved that $B\in 
\mathcal{M}(G)$.

Conversely, let $B\in \mathcal{M}(G)$. Then $\left\vert B\cap
\{x_{j},y_{j}\}\right\vert =1$, which implies that either $B\cap
\{x_{j},y_{j}\}=\{x_{j}\}$ or $B\cap \{x_{j},y_{j}\}=\{y_{j}\}$. If $B\cap
\{x_{j},y_{j}\}=\{x_{j}\}$, then, since $B\cap \{x_{i},y_{j}\}\neq \emptyset 
$, it follows that $x_{i}\in B$. Put $C=B\cap (V_{n}\backslash
\{x_{j},y_{j}\})$. For all $\{x_{k},y_{l}\}\in E(H)$ we have$\ \emptyset
\neq B\cap \{x_{k},y_{l}\}\subset $ $B\backslash \{x_{j}\}=C$, hence $C$ is
a vertex cover of $H$. Since $\left\vert C\right\vert =\left\vert
B\right\vert -1=n-1$, we get $C\in \mathcal{M}(H)$. Similarly, if $B\cap
\{x_{j},y_{j}\}=\{y_{j}\}$, then $x_{j}\notin B$ and $C=B\backslash
\{y_{j}\}\in \mathcal{M}(H)$.
\end{proof}

Herzog and Hibi proved in \cite[Theorem 1.2]{HerHib2} that each unmixed bipartite 
graph on $V_{n}=W\cup W^{\prime }$, $n\geq 1$, can be uniquely associated to a sublattice 
$\mathcal{L}_{G}$ of the Boolean lattice $\mathcal{L}_{n}$ on $\{p_{1},p_{2},...,p_{n}\}$ 
such that $\emptyset \in \mathcal{L}_{G}$ and 
$\{p_{1},p_{2},...,p_{n}\}\in \mathcal{L}_{G}$. The lattice $\mathcal{L}_{G}$ 
is defined as $\{\alpha \subset \{p_{1},p_{2},...,p_{n}\} | C\in \mathcal{M}%
(G),p_{k}\in \alpha $ $\Leftrightarrow $ $x_{k}\in C\}$.

\begin{Corollary}
\label{isolat} In the hypothesis of Lemma \ref{vertcov} and with the above
notation, we have $\mathcal{L}_{H}\simeq \mathcal{L}_{G}{\text{.}}$
\end{Corollary}

\begin{proof}
Let $\nu :\mathcal{L}_{H}\rightarrow \mathcal{L}_{G}$ defined by $\nu
(\alpha )=%
\begin{cases}
\alpha \cup \{p_{j}\}, & \text{if }p_{i}\in \alpha \text{,} \\ 
\alpha , & \text{if }p_{i}\not\in \alpha \text{.}%
\end{cases}%
\ $. By Lemma \ref{vertcov}, $\nu $ is well defined and bijective. It can be
easily checked that $\nu $ is a lattice isomorphism.
\end{proof}

We show that $G$ has a unique Cohen-Macaulay bipartite subgraph, up to a graph 
isomorphism, such that the lattices associated to $G$ and $G^{\prime}$ are isomorphic.

\begin{Proposition}
\label{isol} Let $G$ be an unmixed bipartite graph on $V_{n}=W\cup W^{\prime }$, 
$n\geq 1$, without isolated vertices. Then there exists a Cohen-Macaulay bipartite 
subgraph $G^{\prime }$ of $G$, unique up to a graph isomorphism, such that 
$\mathcal{L}_{G^{\prime }}\simeq \mathcal{L}_{G}$. In particular, all maximal 
chains of $\mathcal{L}_{G}$ have the same length.
\end{Proposition}

\begin{proof}
We proceed by induction on $n$. If $n=1$, then $G$ is Cohen-Macaulay. Put 
$G^{\prime }=G$ and the assertion trivially holds.

Let us suppose that $n>1$. If $G$ is Cohen-Macaulay, then put $G^{\prime }=G$. 
If $G$ is not Cohen-Macaulay, then, by \cite[Theorem 3.4]{CRT}, $G$ has an induced
subgraph $K_{\{i,j\}}$with $1\leq i<j\leq n$. Let $H$ be the subgraph of $G$
induced by the subset $V_{n}\backslash \{x_{j},y_{j}\}\subset V_{n}$. By the
induction hypothesis there exists a unique Cohen-Macaulay bipartite subgraph $%
G^{\prime }$ of $H$, up to a graph isomorphism, such that 
$\mathcal{L}_{G^{\prime }}\simeq \mathcal{L}_{H}$. Obviously, $G^{\prime }$ 
is also a subgraph of $G$. By Corollary \ref{isolat} $\mathcal{L}_{H}\simeq \mathcal{L}_{G}$, 
hence $\mathcal{L}_{G^{\prime }}\simeq \mathcal{L}_{G}$. Since $G^{\prime }$ is
Cohen-Macaulay, it follows, by \cite[Theorem 2.2]{HerHib2}, that $\mathcal{L}%
_{G^{\prime }}$ is a full sublattice of the Boolean lattice on $%
\{p_{i}|x_{i}\in V(G^{\prime })\}$, which implies that all maximal chains of 
$\mathcal{L}_{G^{\prime }}$ have the same length, hence the conclusion.
\end{proof}

\begin{Remark} \emph{Let $G$ be an unmixed bipartite graph on $V_{n}=W\cup W^{\prime }$, 
$n\geq 1$, without isolated vertices. One may derive a procedure to compute a Cohen-Macaulay 
bipartite subgraph $G^{\prime }$ of $G$ such that the lattices $\mathcal{L}_{G}$ and 
$\mathcal{L}_{G^{\prime }}$ are isomorphic. In fact, $G^{\prime }=G_{F}$, where $F$ is a maximal 
subset of $[n]$ such that $K_{\{i,j\}}$ is not an induced subgraph of $G_{F}$, for all distinct 
$i,j\in F$.}
\end{Remark}

\begin{Remark} 
\label{isoalg} \emph{If $G^{\prime }=G_{F}$, $F\subset \lbrack n]$, is a
Cohen-Macaulay bipartite subgraph of $G$ from Proposition \ref{isol}, then, by
Corollary \ref{isolat}, the lattice isomorphism $\nu :\mathcal{L}_{G^{\prime
}}\rightarrow \mathcal{L}_{G}$ is defined by $\nu (\alpha ^{\prime })=\alpha
^{\prime }\cup \{p_{j}|j\in \lbrack n]\backslash F,p_{i}\in \alpha ^{\prime
},K_{\{i,j\}}\subset G\}$, for all $\alpha ^{\prime }\in \mathcal{L}_{G^{\prime }}$.}
\end{Remark}

\section{A  Gr\"{o}bner basis of the toric ideal of the vertex cover algebra of an unmixed bipartite graph}

Let $S=K[x_{1},...,x_{n},y_{1},...,y_{n}]$ and let $G$ be an unmixed bipartite graph on 
$V_{n}=W\cup W^{\prime }$, $n\geq 1$, without isolated vertices. In this case $A(G)$ is the 
Rees algebra of the \textit{cover ideal} $I_{G}$, which is generated by all monomials 
$x_{1}^{c_{1}}\ldots x_{n}^{c_{n}}y_{1}^{c_{n+1}}...y_{n}^{c_{2n}}$, where $c=(c_{1},\ldots
,c_{2n})$ is a $1$-vertex cover of $G$ (\cite{HerHib3}). Thus

\begin{center}
$A(G)=S\oplus I_{G}t\oplus \ldots \oplus I_{G}^{k}t^{k}\oplus \ldots $
\end{center}

\noindent Let $\mathcal{L}_{G}$ be the lattice associated to $G$. Put 
$B_{G}=K[\{x_{i}\}_{1\leq i\leq n},\{y_{j}\}_{1\leq j\leq n},\{u_{\alpha}\}_{\alpha 
\in \mathcal{L}_{G}}]$. For each $\alpha \in \mathcal{L}_{G}$ we denote 
$m_{\alpha}=(\prod\limits_{p_{i}\in \alpha }x_{i})\cdot (\prod\limits_{p_{j}\not\in \alpha }y_{j})$. 
The \textit{toric ideal} $Q_{G}$ of $A(G)$ is the kernel of the surjective ring homomorphism 
$\xi :B_{G}\rightarrow {A}(G)$, $\xi (x_{i})=x_{i}$, $\xi (y_{i})=y_{i}$, $1\leq i\leq n$, 
$\xi (u_{\alpha })=m_{\alpha }t$, $\alpha\in\mathcal{L}_{G}$.

Let $<_{lex}$ the lexicographic order on $S$ induced by the ordering of the variables 
$x_{1}>...>x_{n}>y_{1}>...>y_{n}$. Let $<^{\#}$ the reverse lexicographic
order on the polynomial ring $K[\{u_{\alpha }\}_{\alpha \in \mathcal{L}%
_{G}}] $ induced by an ordering of the variables $u_{\alpha }$'s such that $%
u_{\alpha }>u_{\beta }$ if $\beta \subset \alpha $ in $\mathcal{L}_{G}$. Let 
$<_{lex}^{\sharp }$ the monomial order on $B_{G}$ defined as the product of
the monomial orders $<_{lex}$ and $<^{\#}$ from above. This monomial order
was introduced in \cite{HerHib1}.

Next, inspired by \cite[Theorem 1.1]{HerHib1}, we compute the reduced Gr\"{o}bner basis 
of the toric ideal of the vertex cover algebra of an unmixed bipartite graph $G$ on 
$V_{n}=W\cup W^{\prime }$, $n\geq 1$, with respect to the monomial order $<_{lex}^{\sharp }$. 
For $\alpha \in \mathcal{L}_{G}$ let $V(\alpha )$ be the set of 
all upper neighbours of $\alpha $ in $\mathcal{L}_{G}$. We denote 
$x_{\beta \backslash \alpha }=\prod\limits_{p_{i}\in \beta
\backslash \alpha }x_{i}$, $y_{\beta \backslash \alpha
}=\prod\limits_{p_{i}\in \beta \backslash \alpha }y_{i}$, where 
$\alpha \in \mathcal{L}_{G}$, $\alpha \not=\{p_{1},p_{2},..,p_{n}\}$ and $\beta \in V(\alpha )$.

\begin{Theorem}
\label{Grobas} Let $G$ be an unmixed bipartite graph on $V_{n}=W\cup W^{\prime }$, 
$n\geq 1$, without isolated vertices. Then the reduced Gr\"{o}bner basis of the toric ideal $%
Q_{G}$ of the vertex cover algebra $A(G)$ with respect to $<_{lex}^{\sharp }$
is:

\begin{center}
$\mathcal{G}_{<_{lex}^{\#}}(Q_{G})=\{\underline{x_{\beta \backslash \alpha }u_{\alpha }}%
-y_{\beta \backslash \alpha }u_{\beta }$ $|$ $\alpha \in \mathcal{L}%
_{G},\alpha \not=\{p_{1},p_{2},..,p_{n}\}$, $\beta \in V(\alpha )\}$

$\cup \{\underline{u_{\alpha }u_{\beta }}-u_{\alpha \cup \beta }u_{\alpha
\cap \beta }$ $|\alpha ,\beta \in $ $\mathcal{L}_{G},\alpha \not\subset
\beta ,\beta \not\subset \alpha \}$,
\end{center}

\noindent where the initial monomial of each binomial is the first monomial.
\end{Theorem}

\begin{proof}
We essentially follow the proof of \cite[Theorem 1.1]{HerHib1} with a slight modification 
in its last part. As it was shown there, we only need to consider a primitive binomial $g$ of 
the reduced Gr\"{o}bner basis of $Q_{G}$ with respect to $<_{lex}^{\sharp }$. Let $g\in \mathcal{G}_{<_{lex}^{\#}}(Q_{G})$, 
\[
g=\left( \prod\limits_{i=1}^{n}x_{i}^{a_{i}}y_{i}^{b_{i}}\right) \left(
u_{\alpha _{1}}u_{\alpha _{2}}...u_{\alpha _{r}}\right) -\left(
\prod\limits_{i=1}^{n}x_{i}^{a_{i}^{\prime }}y_{i}^{b_{i}^{\prime }}\right)
\left( u_{\alpha _{1}^{\prime }}u_{\alpha _{2}^{\prime }}...u_{\alpha_{r}^{\prime }}\right),
\]
with $\alpha _{1}\subsetneqq \alpha _{2}\subsetneqq ...\subsetneqq \alpha _{r}$
and $\alpha _{1}^{\prime }\subsetneqq \alpha _{2}^{\prime }\subsetneqq
...\subsetneqq \alpha _{r}^{\prime }$ chains in $\mathcal{L}_{G}$ and $\ini_{<_{lex}^{\#}}(g)$ 
equal to the first monomial of $g$. 

We assume that $g\notin K[\{u_{\alpha }\}_{\alpha \in \mathcal{L}_{G}}]$. As in 
\cite[Theorem 1.1]{HerHib1} we get that there exists some $1\leq j\leq r$ such that 
$\alpha _{j}^{\prime }\not\subset \alpha _{j}$. Let $\beta$ be an upper neighbour of $\alpha _{j}$ with 
$\alpha _{j}\subset \beta\subset \alpha _{j}\cup \alpha _{j}^{\prime }$ and let 
$p_{i}\in \beta \backslash \alpha _{j}$. Then $p_{i}\in \alpha _{k}^{\prime }$ for all $k\geq j$ and $p_{i}\notin \alpha_{l}$ for all $l\leq j$. This implies that $a_{i}>0$ for all $i$ for which  
$p_{i}\in \beta \backslash \alpha _{j}$. Then the binomial 
$h=x_{\beta\backslash \alpha_{j}}u_{\alpha _{j}}-y_{\beta\backslash \alpha _{j}}u_{\beta}\in Q_{G}$ and 
$\ini_{<_{lex}^{\#}}(h)=x_{\beta \backslash \alpha _{j}}u_{\alpha_{j}}|\ini_{<_{lex}^{\#}}(g)$. 
Hence $\ini_{<_{lex}^{\#}}(g)$ must coincide with $x_{\beta \backslash \alpha _{j}}u_{\alpha_{j}}$, 
and, moreover, $h=g$.
\end{proof}

\section{The Hilbert series of the vertex cover algebra of unmixed bipartite graphs}

Let $\{m_{1},m_{2},...,m_{l}\}$ be the minimal system of generators of $%
I_{G} $. We view $A(G)$ as a standard graded $K$-algebra by assigning to each 
$x_{i} $ and $y_{j}$, $1\leq i,j\leq n$ and to each $m_{k}t$, $1\leq k\leq l$%
, the degree $1$.

The Hilbert function and the Hilbert series of the vertex cover algebra $%
A(G) $ are invariant to a certain class of graph isomorphisms.

\begin{Remark}
\label{iso} \emph{Let $\sigma $ be a permutation of $[n]$ and let 
$^{\sigma}G$ denote the bipartite graph on $V_{n}=W\cup W^{\prime }$ with the edge set $E(^{\sigma
}G)=\{\{x_{\sigma (i)},y_{\sigma (j)}\}|\{x_{i},y_{j}\}\in E(G)\}$. The graph isomorphism 
$h:V(G)\rightarrow V(^{\sigma }G)$, 
$h(x_{i})=x_{\sigma (i)}$ and $h(y_{j})=y_{\sigma (j)}$, $i,j\in \lbrack n]$, induces a 
$K$-automorphism of $S$ which maps $I_{G}$ onto $I_{^{\sigma }G}$. Therefore, $A(G)$ and 
$A(^{\sigma }G)$ have the same Hilbert function and series.}
\end{Remark}

Let $\Delta (\mathcal{L}_{G})$ be the \textit{order complex} of the lattice $%
\mathcal{L}_{G}$. (We refer the reader to \cite[\S 5.1]{BruHer} for the
definition and properties of the order complex associated to a poset.) Let $%
S_{G}=K[\{u_{\alpha }\}_{\alpha \in \mathcal{L}_{G}}]$ be the polynomial
ring in $\left\vert \mathcal{L}_{G}\right\vert $ variables over $K$. The
toric ideal $\bar{Q}_{G}$ of the basic cover algebra $\bar{A}(G)$ is the
kernel of the surjective ring homomorphism $\pi :S_{G}\rightarrow \bar{A}(G)$
defined by $\pi (u_{\alpha })=m_{\alpha }$, for all $\alpha \in \mathcal{L}_{G}$.

\begin{Proposition}
\label{vect} The graded $K$-algebra $\bar{A}(G)$ and the order complex $%
\Delta (\mathcal{L}_{G})$ have the same vector $h$-vector.
\end{Proposition}

\begin{proof}
By \cite[Proposition 3.1]{HerHib2} $\bar{Q}_{G}$ is a graded ideal 
and the initial ideal $\ini_{<^{\#}}(\bar{Q}_{G}) $
of the toric ideal $\bar{Q}_{G}$ coincides with the Stanley-Reisner ideal $%
I_{\Delta (\mathcal{L}_{G})}$, hence $S_{G}/\bar{Q}_{G}$ and $K[\Delta (%
\mathcal{L}_{G})]$ have the same $h$-vector. Since $S_{G}/\bar{Q}_{G}$ 
and $\bar{A}(G)$ are isomorphic as graded $K$-algebras, the conclusion follows.
\end{proof}

\begin{Remark}
\label{hvector} \emph{Let $G$ be an unmixed bipartite graph on $V_{n}=W\cup W^{\prime }$, 
$n\geq 1$, without isolated vertices and let $G^{\prime }$ be a Cohen-Macaulay bipartite 
subgraph of $G$ with $\mathcal{L}_{G}\simeq \mathcal{L}_{G^{\prime }}$. If $h$, respectively 
$h^{\prime }$, are the $h$-vectors of $\bar{A}(G)$, respectively $\bar{A}(G^{\prime })$, then,
by using Proposition \ref{vect} and the fact that the lattices $\mathcal{L}_{G}$ and 
$\mathcal{L}_{G^{\prime }}$ are isomorphic, it follows that $h=h^{\prime }$. Moreover, 
by \cite[Remark 1.3]{Ion}, $h_{i}\geq 0$, for all $0\leq i\leq r+1$ and $h_{r}=h_{r+1}=0$, 
where $r=rank(\mathcal{L}_{G})$.}
\end{Remark}

In order to prove the main theorem we need some preparatory results. They are closely 
related to those for the Cohen-Macaulay case which were proved in \cite{Ion}.

Let $\emptyset \neq F\subsetneqq \lbrack n]$, $P_{n}(F)=\{p_{i}|i\in F\}$ 
and let $\alpha \in \mathcal{L}_{G_{\bar{F}}}$, where $\bar{F}$ denotes 
the complemet set of $F$ in $[n]$. We denote by $\delta _{\alpha }$
the maximal subset of $P_{n}(F)$ such that $\alpha \cup \delta _{\alpha }\in 
\mathcal{L}_{G}$. Note that 
\begin{equation*}
\delta _{\alpha }=\cup \{\gamma \ |\ \gamma \subset P_{n}(F),\alpha \cup
\gamma \in \mathcal{L}_{G}\}.
\end{equation*}%
If we set $\beta =\alpha \cup \delta _{\alpha },$ then, by the definition of 
$\delta _{\alpha },$ $\beta $ has the following property: there exists no
subset $\emptyset \neq A\subset F$ such that $\beta \cup \{p_{i}|i\in A\}$
is an upper neighbour of $\beta $ in $\mathcal{L}_{G}$.

\begin{Lemma}
\label{lemadelta} Let $\emptyset \neq F\subsetneqq \lbrack n]$ and let $%
\mathcal{S}$ be the set of all $\beta \in \mathcal{L}_{G}$ with the property
that there exists no subset $\emptyset \neq A\subset F$ such that $\beta
\cup \{p_{i}|i\in A\}$ is an upper neighbour of $\beta $ in $\mathcal{L}(G)$%
. Then the map $\varphi \colon \mathcal{L}_{G_{\bar{F}}}\rightarrow \mathcal{%
S}$ defined by $\alpha \mapsto \beta =\alpha \cup \delta _{\alpha }$, is an
isomorphism of posets.
\end{Lemma}

\begin{proof}
We follow the proof of Lemma 1.4 in \cite{Ion}. We show that $\varphi $ is 
invertible. Indeed, the map $\psi :\mathcal{S}\rightarrow \mathcal{L}(G_{\bar{F}})$ 
defined by $\psi (\beta )=\beta \cap P_{n}(\bar{F})$ is the inverse of 
$\varphi $ since $\alpha =\beta \cap P_{n}(\bar{F})$ and 
$\alpha \in \mathcal{L}(G_{\bar{F}})$.

Let $\ \alpha _{1},\alpha _{2}\in \mathcal{L}(G_{\bar{F}})$ with $\alpha
_{1}\subsetneqq \alpha _{2}$ and $\beta _{i}=\varphi (\alpha _{i})=\alpha
_{i}\cup \delta _{i},\ i=1,2.$ We only need to show that $\beta _{1}\subset
\beta _{2}$ since the strict inclusion follows from the hypothesis $\alpha
_{1}\subsetneqq \alpha _{2}$. Let us assume that $\beta _{1}\not\subset \beta
_{2}$ and let $p_{r_{1}}\in \beta _{1}\backslash \beta _{2}$ with $r_{1}\in
F $. Since $p_{r_{1}}\notin \delta _{2}$, it follows that $\beta _{2}\cup
\{p_{r_{1}}\}\notin \mathcal{L}_{G}$.

We claim that $\{u\in F|p_{u}\in \beta _{1}\backslash \{p_{r_{1}}\}\}\neq
\emptyset $. Let us suppose, on the contrary, that $p_{u}\not\in \beta
_{1}\backslash \{p_{r_{1}}\}$, for all $u\in F$. Then $\beta _{1}=\alpha
_{1}\cup \{p_{r_{1}}\}$. Since $\beta _{1},\beta _{2}\in \mathcal{L}_{G}$,
it follows that $\beta _{1}\cup \beta _{2}\in \mathcal{L}_{G}$. On the other
hand, we have $\beta _{1}\cup \beta _{2}=\beta _{2}\cup \{p_{r_{1}}\}$,
which implies that $\beta _{1}\cup \beta _{2}\notin \mathcal{L}_{G}$, a
contradiction.

By repeated application of this argument we get the sequence 
$r_{1},r_{2},...,r_{k},r_{k+1},...$ with $r_{k+1}\in F\backslash
\{r_{1},...,r_{k}\}$ and $p_{r_{k+1}}\in \beta _{1}\backslash (\beta _{2}\cup
\{p_{r_{1}},...,p_{r_{k}}\})$, for all $k \geq 0 $. Therefore, the set $F$ 
is infinite, which is impossible. Hence $\beta _{1}\subset \beta _{2}$.

Now let $\beta _{1},\beta _{2}\in \mathcal{S}$ with $\beta _{1}\subsetneqq 
\beta _{2}$ and assume that $\alpha _{1}=\alpha _{2},$ where $\alpha
_{1}=\beta _{1}\cap P_{n}(\bar{F}),$ and $\alpha _{2}=\beta _{2}\cap P_{n}(%
\bar{F}).$ Then $\delta _{1}=\beta _{1}\setminus P_{n}(\bar{F})\subsetneqq
\delta _{2}=\beta _{2}\setminus P_{n}(\bar{F}).$ But this is impossible
since $\delta _{1}$ is maximal among the subsets $\gamma \subset P_{n}(F)$ 
such that $\alpha _{1}\cup \gamma \in \mathcal{L}_{G}$.
\end{proof}

The next result relates the Hilbert series of the vertex cover
algebras $A(G)$ to the Hilbert series of the basic covers algebras $\bar{A}(G_{F})$, 
for all $F\subset \lbrack n]$. If $F=\emptyset $, we put by convention 
$H_{\bar{A}(G_{F})}(z)=\frac{1}{1-z}$. 

\begin{Theorem}
\label{comp} Let $G$ be an unmixed bipartite graph on $V_{n}=W\cup W^{\prime }$, 
$n\geq 1$, without isolated vertices. For $F\subset \lbrack n]$ let 
$r_{F}=\rank(\mathcal{L}_{G_{F}})$, let $H_{\bar{A}(G_{F})}(z)$ be the Hilbert series 
of $\bar{A}(G_{F})$, and $H_{A(G)}(z)$ be the Hilbert series of $A(G)$. Then:

\begin{equation}
H_{A(G)}(z)=\frac{1}{(1-z)^{n}}\sum\limits_{F\subset \lbrack n]}H_{\bar{A}%
(G_{F})}(z)\left( \frac{z}{1-z}\right) ^{n-\left\vert F\right\vert }\text{.}
\label{equ6}
\end{equation}

In particular, if $h(z)=\sum\limits_{j\geq 0}h_{j}z^{j}$, respectively $%
h^{F}(z)=\sum\limits_{j\geq 0}h_{j}^{F}z^{j}$, where $h=(h_{j})_{j\geq 0}$, 
respectively $h^{F}=(h_{j}^{F})_{j\geq 0}$, are the $h$-vectors of $A(G)$, 
respectively of $\bar{A}(G_{F})$, then 
\begin{equation}
h(z)=\sum\limits_{F\subset \lbrack n]}h^{F}(z)(1-z)^{\left\vert F\right\vert
-r_{F}}z^{n-\left\vert F\right\vert }\text{.}  \label{equ7}
\end{equation}
\end{Theorem}

\begin{proof}
(\ref{equ6}) can be proved exactly as in \cite[Theorem 1.5]{Ion}.

It is known that $H_{A(G)}(z)=\frac{h(z)}{(1-z)^{2n+1}}$ (since dim$A(G)=$ dim$S+1=2n+1$ \cite{BruHer}) and $H_{\bar{A}(G_{F})}=\frac{h^{F}(z)}{(1-z)^{r_{F}+1}}$ (since dim$\bar{A}(G_{F})=r_{F}+1$ \cite{BerMic}), 
for all $F\subset \lbrack n]$, hence 
$h(z)=\sum\limits_{F\subset \lbrack n]}h^{F}(z)(1-z)^{\left\vert F\right\vert
-r_{F}}z^{n-\left\vert F\right\vert }$.
\end{proof}

\begin{Remark} \emph{
By using (\ref{equ7}) we get}\\

\begin{center}
$h_{n+1}=\sum\limits_{F\subset \lbrack n]}(-1)^{\left\vert F\right\vert
-r}h_{r+1}^{F}$ \emph{and} $h_{n}=\sum\limits_{F\subset \lbrack n]}(-1)^{\left\vert F\right\vert
-r}[h_{r}^{F}-(\left\vert F\right\vert -r)h_{r+1}^{F}]$,
\end{center}

\noindent \emph{where $r=r_{F}=\rank(\mathcal{L}_{G_{F}})$. By Remark \ref{hvector},  
$h_{r}^{F}=h_{r+1}^{F}=0$, for all $\emptyset \neq F\subset \lbrack n]$. 
Hence $h_{n+1}=h_{1}^{\emptyset }=0$, $h_{n}=h_{0}^{\emptyset }=1$ and the $a$-invariant 
of $A(G)$ is $a=-n-1$. In \cite[Corollary 4.4]{HerHib3} 
it was proved that $A(G)$ is a Gorenstein ring, therefore, by 
\cite[Corollary 4.3.8 (b) and Remark 4.3.9 (a)]{BruHer},  
$h_{i}=h_{n-i}$, for all $0\leq i\leq n$.}
\end{Remark}

\begin{Corollary}
\label{multi} Let $G$ be an unmixed bipartite graph on $V_{n}=W\cup W^{\prime }$, 
$n\geq 1$, without isolated vertices. Then

\begin{equation}
e(A(G))=\sum\limits_{\substack{ F\subset \lbrack n] \\ G_{F}\text{
Cohen-Macaulay}}}e(\bar{A}(G_{F}))\text{,}  \label{equ11}
\end{equation}

\noindent where, by convention, $G_{\emptyset }$ is considered a Cohen-Macaulay subgraph
of $G$.
\end{Corollary}

\begin{proof}
By \cite[Theorem 2.2]{HerHib2} $G_{F}$ is a Cohen-Macaulay bipartite graph if and only 
if $\rank(\mathcal{L}_{G_{F}})=\left\vert F\right\vert $, for all 
$\emptyset \neq F\subset \lbrack n]$. Thus (\ref{equ11}) follows immediately from (\ref{equ7}).
\end{proof}

We compute the Hilbert series of the vertex cover algebra of unmixed
complete bipartite graphs $K_{n,n}$, $n\geq 1$.

\begin{Proposition}
\label{complet} For all $n\geq 1$ $H_{A(K_{n,n})}(z)=\frac{1+z+...+z^{n}}{%
(1-z)^{2n+1}}$. In particular, the multiplicity $e(A(K_{n,n}))=n+1$.
\end{Proposition}

\begin{proof}
$\mathcal{L}_{K_{n,n}}=\{\emptyset ,\{p_{1},p_{2},...,p_{n}\}\}$, therefore,
by Theorem \ref{Grobas},\ $Q_{K_{n,n}}$ is a principal ideal generated by $%
b=x_{1}...x_{n}u_{2}-y_{1}...y_{n}u_{1}$, where $u_{1}=u_{\{p_{1},...,p_{n}\}}$ 
and $u_{2}=u_{\emptyset }$. Then $A(K_{n,n})\simeq B_{K_{n,n}}/Q_{K_{n,n}}$ and the minimal 
graded free resolution of $A(K_{n,n})$ is given by the exact short sequence 
$0\rightarrow B_{K_{n,n}}(-(n+1))\overset{\cdot \text{ }b}{\rightarrow }%
B_{K_{n,n}}\rightarrow A(K_{n,n})\rightarrow 0$. Hence $H_{A(K_{n,n})}(z)=%
\frac{1+z+...+z^{n}}{\left( 1-z\right) ^{2n+1}}$. In particular, the multiplicity 
$e(A(K_{n,n}))=n+1$.
\end{proof}

Let $P_{n}=\{p_{1},p_{2},\ldots ,p_{n}\}$ be a poset with a partial order $\leq $. We denote 
by $G(P_{n})$ the bipartite graph on $V_{n}=W\cup W^{\prime }$, whose edge set 
$E(G)$ consists of all $2$-element subsets $\{x_{i},y_{j}\}$ with $p_{i}\leq p_{j}$. It is said 
that a bipartite graph $G$ on $V_{n}=W\cup W^{\prime }$ \textit{comes from a poset}, if there 
exists a finite poset $P_{n}$ on $\{p_{1},p_{2},\ldots ,p_{n}\}$ such that $p_{i}\leq p_{j}$ 
implies $i\leq j$, and after relabeling of the vertices of $G$ one has $G=G(P_{n})$. 

\begin{Corollary}
\label{ineq} Let $G$ be an unmixed bipartite graph on $V_{n}=W\cup W^{\prime }$, 
$n\geq 1$, without isolated vertices. Then 
\begin{equation*}
n+1\leq e(A(G))\leq n!\sum\limits_{l=0}^{n}\frac{1}{l!}\text{.}
\end{equation*}%
The left equality holds if and only if $G=K_{n,n}$ and the right equality
holds if and only if $G$ comes from an antichain with $n$ elements.
\end{Corollary}

\begin{proof}
By (\ref{equ11}) and \cite[Proposition 3.4(3)]{BenCon},  
$e(A(G))=\sum\limits_{\substack{ F\subset \lbrack n]  \\ r=\left\vert F\right\vert }}f_{r}^{F}$, 
where $r=\rank(\mathcal{L}_{G_{F}})$ and $f_{r}^{F}$ is the last component of the $f$%
-vector of the order complex $\Delta (\mathcal{L}_{G_{F}})$. Then $%
e(A(G))=n+1+\sum\limits_{\substack{ F\subset \lbrack n]  \\ r=\left\vert
F\right\vert \geq 2}}f_{r}^{F}$, which implies that $e(A(G))\geq n+1$. The
equality holds if and only if $\rank(\mathcal{L}_{G_{F}})<\left\vert
F\right\vert $ for all $F\subset \lbrack n]$ with $\left\vert F\right\vert
\geq 2$, which is equivalent to $G=K_{n,n}$. On the other hand, $%
e(A(G))=1+\sum\limits_{\substack{ F\subset \lbrack n]  \\ r=\left\vert
F\right\vert \geq 1}}f_{r}^{F}$. If $\rank(\mathcal{L}_{G_{F}})=\left\vert
F\right\vert $, $\emptyset \neq F\subset \lbrack n]$, then $\mathcal{L}_{G_{F}}$ 
is a full sublattice of a Boolean lattice on a set with $%
\left\vert F\right\vert $ elements, hence $f_{\left\vert F\right\vert
}^{F}\leq \left\vert F\right\vert !$ and $e(A(G))\leq 1+\sum\limits 
_{\substack{ F\subset \lbrack n]  \\ \left\vert F\right\vert \geq 1}}\binom{n%
}{\left\vert F\right\vert }\left\vert F\right\vert !=n!\sum\limits_{l=0}^{n}%
\frac{1}{l!}$. The equality holds if and only if $\mathcal{L}_{G_{F}}$ is a
Boolean lattice on a set with $\left\vert F\right\vert $ elements, for all $%
\emptyset \neq F\subset \lbrack n]$, which is equivalent to saying that $G$ comes from an antichain.
\end{proof}

\begin{Remark} \emph{In general, unmixed bipartite graphs are not uniquely determined, 
up to an isomorphism, by the $h$-vector of their corresponding vertex cover algebras.
Let $G_{3}$ be the bipartite graph on $V_{3}$ with the edge set: $%
\{x_{1},y_{1}\}$, $\{x_{2},y_{2}\}$, $\{x_{3},y_{3}\}$, $\{x_{2},y_{3}\}$, $%
\{x_{3},y_{2}\}$ and $G_{3}^{\prime }$ be the bipartite graph on $V_{3}$
that comes from the chain $P_{3}^{\prime }=\{p_{1}^{\prime },p_{2}^{\prime
},p_{3}^{\prime }\}$ with $p_{1}^{\prime }\leq p_{2}^{\prime }\leq
p_{3}^{\prime }$. $G_{3}$ and $G_{3}^{\prime }$ are unmixed and they are not
isomorphic, and by computation we get $H_{A(G_{3})}(z)=H_{A(G_{3}^{\prime
})}(z)=\frac{1+3z+3z^{2}+z^{3}}{(1-z)^{7}} $.}
\end{Remark}

However, unmixed complete bipartite graphs and bipartite graphs that come
from chains and antichains are uniquely determined (up to a graph
isomorphism) by the $h$-vector of their corresponding vertex cover algebras. 
The statement for bipartite graphs that come from chains and antichains was proved 
in \cite[Proposition 2.3]{Ion}.

\begin{Corollary}
Let $G$ be an unmixed bipartite graph on $V_{n}=W\cup W^{\prime }$, $n\geq 1$, 
without isolated vertices. Then $G=K_{n,n}$ if and only if 
$H_{A(G)}(z)=\frac{1+z+...+z^{n}}{(1-z)^{2n+1}}$.
\end{Corollary}

\begin{proof}
($"\mathbf{If}"$) By using (\ref{equ7}) we get $h_{1}=h_{1}^{[n]}+rank(%
\mathcal{L}_{G}).$ Since $h_{1}^{[n]}$ is the component of rank $1$ in the $%
h $-vector of $\bar{A}(G)$, by using the formula which relates the $h$%
-vector to the $f$-vector of the order complex $\Delta (\mathcal{L}_{G})$, 
we get $h_{1}^{[n]}=\left\vert \mathcal{L}%
_{G}\right\vert -rank(\mathcal{L}_{G})-1,$ which implies that $%
h_{1}=\left\vert \mathcal{L}_{G}\right\vert -1$. By hypothesis, $h_{1}=1$,
hence $\left\vert \mathcal{L}_{G}\right\vert =2$. $G$ is an unmixed
bipartite graph on $V_{n}$, therefore, $\mathcal{L}_{G}=\{\emptyset
,\{p_{1},...,p_{n}\}\}$ and $G=K_{n,n}$.

($"\mathbf{Only}$\textbf{\ }$\mathbf{if}"$) It follows from Proposition \ref%
{complet}.
\end{proof}

\section*{Acknowledgment}

I would like to  thank Professor J\"{u}rgen Herzog for very useful
 discussions on the subject of this paper.

\end{document}